\documentclass[12pt, letterpaper]{article}
\usepackage{amssymb}
\usepackage{amsthm}
\usepackage{amsmath}
\usepackage{graphicx}
\usepackage{hyperref}

\newtheorem{thm}{Theorem}

\newtheorem{lem}[thm]{Lemma}
\theoremstyle{definition}
\newtheorem{REM}{Remark}

\begin{document}
\title{Enumerating (multiplex) juggling sequences}
\author{Steve Butler\thanks{{\tt sbutler@math.ucsd.edu}, Department of Mathematics, University of California, San Diego}\and Ron Graham\thanks{{\tt rgraham@cs.ucsd.edu}, Department of Computer Science \& Engineering, University of California, San Diego}}
\date{\empty}
\maketitle
  
\begin{abstract}
We consider the problem of enumerating periodic $\sigma$-juggling
sequences  of length $n$ for multiplex juggling, where $\sigma$ is
the initial state (or {\em landing schedule}) of the balls.  We
first show that this problem is equivalent to choosing $1$'s in a
specified matrix to guarantee certain column and row sums, and
then using this matrix, derive a recursion.  This work is a
generalization of earlier work of Fan Chung and Ron Graham.
\end{abstract}

\section{Introduction}
Starting about 20 years ago, there has been increasing activity by
discrete mathematicians and (mathematically inclined) jugglers in
developing and exploring ways of representing various possible
juggling patterns numerically (e.g., see \cite{BEGW,bg1,bg2,cg,cg1,er, polster,stadler,war}). Perhaps the most prominent of these
is the idea of a {\it juggling sequence}\/ (or ``siteswap'', as it
is often referred to in the juggling literature). The idea behind
this approach is the following. For a given sequence $T = (t_1,
t_2, \ldots, t_n)$ of nonnegative integers, we associate a
(possible) periodic juggling pattern in which at time $i$, a ball
is thrown so that it comes down at time $i + t_i$. This is to be
true for each $i, 1 \leq i \leq n$. Because we assume
this is to be repeated indefinitely with period $n$, then in
general, for each $i$ and each $k \geq 0$, a ball thrown at
time $i + kn$ will come down at time $i + t_i + kn$. The usual
assumption made for a sequence $T$ to be a valid juggling sequence
is that at no time do two balls come down at the same time. This
assumption results in many consequences, e.g., all of the
quantities $i + t_i \pmod n$ must be distinct, the number of balls
in the pattern is the average $({1}/{n}) \sum_{k=1}^{n} t_k$, and
the number of juggling sequences with period $n$ having fewer
than $b$ balls is $b^n$ (see \cite{BEGW}).

An important object for understanding the relationships and
transitions between various juggling sequences is the concept of a
{\it state diagram}, developed independently (and almost 
simultaneously) by Jack Boyce and Allen Knutson
\cite{jis}. This is a directed graph where each vertex is called a
{\it state}\/ or {\it landing schedule}, a $0$-$1$ vector indicating
when the balls that are currently in the air will land, and edges 
represent possible transitions between states.  The vertex and edge
sets for the state diagram can be defined as follows:
\begin{eqnarray*}
V&=&\{\langle a_1,a_2,a_3,\ldots\rangle: a_i\in \{0,1\},
{\textstyle \sum_{i}}a_i=b\},\\
E&=&\{\langle a_1,a_2,a_3,\ldots\rangle{\to}\langle b_1,b_2,b_3,\ldots\rangle :
a_i\leq b_{i-1}\mbox{ for }i=2,3,\ldots\}.
\end{eqnarray*}

More specifically, each juggling sequence $T$ is associated with a state $\sigma=
\sigma_T= \langle \sigma_1, \sigma_2, \ldots, \sigma_h, \ldots
\rangle$ which can be found by imagining that the sequence has
been executed infinitely often in the past, with a final throw
$t_n$ being made at time $0$. Then $\sigma_i$ is $1$ if and only
if there is some ball still in the air at time $0$ that will
land at time $i$.  In this case we
say that $T$ is a $\sigma$-juggling sequence.

If we are now going to throw one more ball at time $1$,
transitioning to a (possibly) new juggling state $\sigma '$, then
we are restricted to throwing it so that it lands at some time $j$
which has $\sigma_j = 0$. The new state $\sigma ' = \langle
\sigma_1 ', \sigma_2 ', \ldots \rangle$ then has $\sigma_k ' =
\sigma_{k+1}$ for $k \geq 1, k \neq j-1$ and $\sigma_{j-1} ' = 1$.
The preceding remarks assume that $\sigma_1 = 1$. If $\sigma_1 =
0$, so that there is no ball available to be thrown at time $1$,
then a ``no-throw'' occurs, and the new state vector $\sigma '$
satisfies $\sigma_k ' = \sigma_{k+1}$ for all $k \geq 1$. These
give the two basic transitions that can occur in the state diagram.
With this interpretation, it is easy to see that a juggling sequence
of period $n$ exactly corresponds to a walk of length $n$ in the state
diagram.

In \cite{cg1}, the problem of enumerating $\sigma$-juggling
sequences of period $n$ was studied, which by the above comments
is equivalent to counting the number of directed closed walks of
length $n$ starting at $\sigma$ in the state diagram. In the same
paper, the related problem of counting the number of ``primitive''
closed walks of length $n$ was also solved. These are walks in
which the starting state $\sigma$ is visited only at the beginning
and the end of the walk.

A particular unsolved problem mentioned in \cite{cg1} was that of
extending the analysis to the much more complex situation of {\it
multiplex}\/ juggling sequences. In a multiplex juggling sequence,
for a given parameter $m$, at each time instance up to $m$ balls
can be thrown and caught at the same time, where the balls thrown
at each time can have different landing times. Thus, ordinary
juggling sequences correspond to the case $m = 1$.

As before, we can describe a (multiplex) juggling sequence as a
walk in a state diagram. Here a state $\alpha=\langle
a_1,a_2,a_3,\ldots\rangle$ can again be described as a landing
schedule where $a_i$ are the number of balls currently scheduled
to land at time $i$.  We also have a state diagram which has as its
vertices all possible states and for edges all ways to go from one
state to another state (see \cite{polster}).  The state diagram is
thus a directed graph with two important parameters: $b$, the
number of balls that are being juggled, and $m$, the maximum
number of balls that can be caught/thrown at any one time.  The
vertex set and edge set are defined as follows:
\begin{eqnarray*}
V&=&\{\langle a_1,a_2,a_3,\ldots\rangle: a_i\in \{0,1,\ldots,m\},
{\textstyle \sum_{i}}a_i=b\},\\
E&=&\{\langle a_1,a_2,a_3,\ldots\rangle{\to}\langle b_1,b_2,b_3,\ldots\rangle :
a_i\leq b_{i-1}\mbox{ for }i=2,3,\ldots\}.
\end{eqnarray*}

Since each state will only have finitely many nonzero terms, we will
truncate the terminal zero portions of the state vectors
when convenient.   The height of a
state $\alpha$ will be the largest index $i$ for which $a_i>0$,
and will be denoted by $h(\alpha)$.  A small portion of the state diagram
when $b=3$ and $m=2$ is shown in Figure~\ref{fig:state}.

\begin{figure}[htb]
\centering
\includegraphics[scale=1.25]{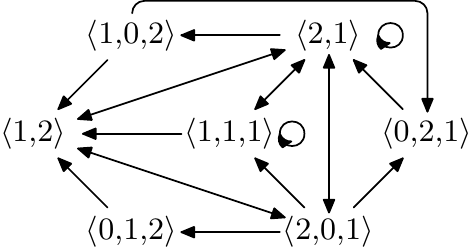}
\caption{A portion of the state diagram when $b=3$ and $m=2$.}
\label{fig:state}
\end{figure}

\subsection{A bucket approach}
To better follow the analysis (and make practical demonstrations easier)
we can reinterpret multiplex juggling by a series of buckets and balls.
The buckets will represent future landing times for $i=1,2,3,\ldots$ and
the balls are distributed among these buckets.  A state vector is then a
listing of how many balls are currently in each bucket, and $m$ is now
the maximum number of balls that can fit inside of a bucket.  Transitions
from state to state happen by having the buckets shift down by one and
redistributing any balls that were in the bottom bucket.  This process
is shown in Figure~\ref{fig:buckets}.

\begin{figure}[htb]
\centering
\includegraphics[scale=.9]{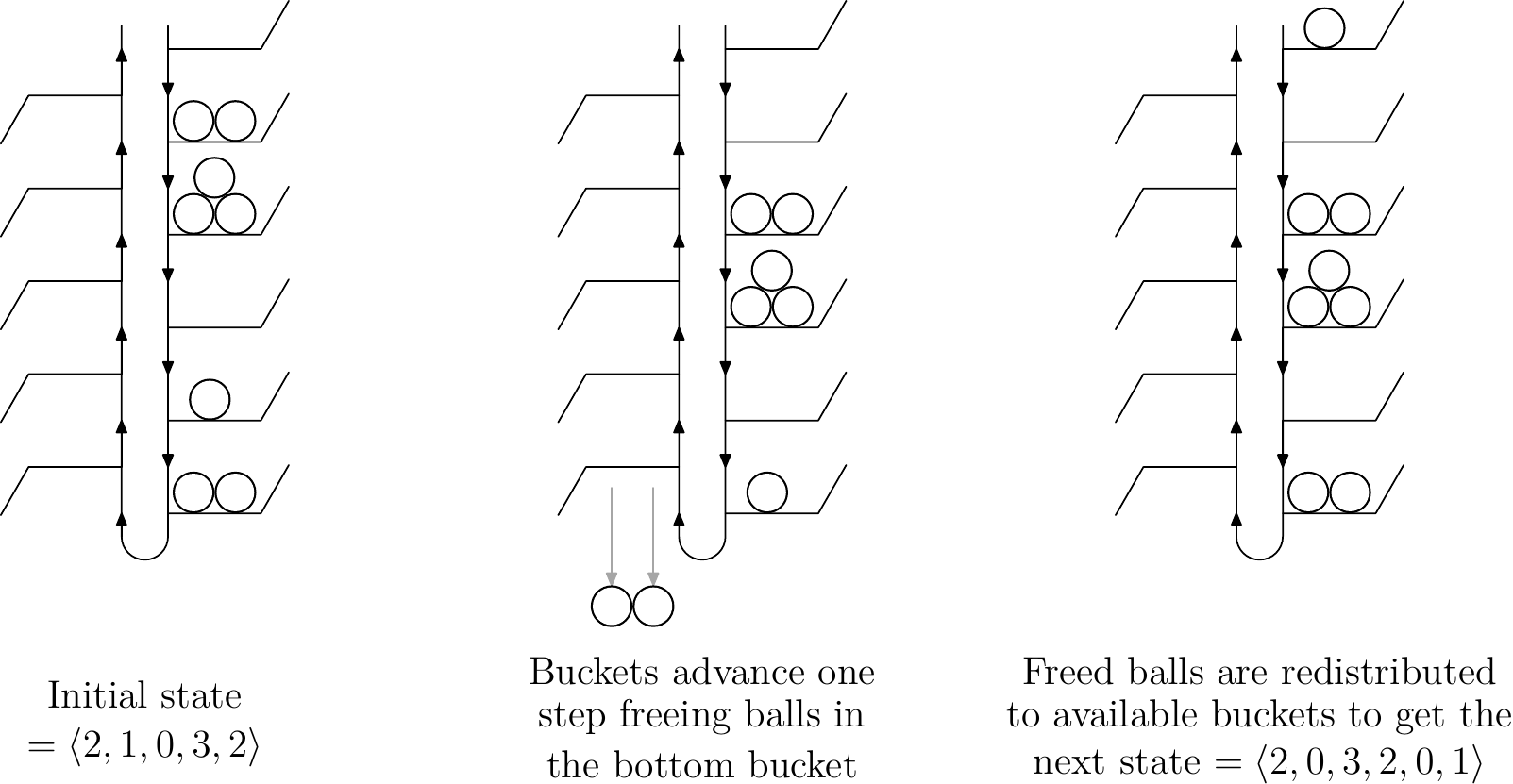}
\caption{A bucket approach to multiplex juggling.}
\label{fig:buckets}
\end{figure}

\subsection{Multiplex siteswap notation}
To describe a walk in the state diagram it suffices to know what
state we start in and how we transition from state to state.  In
transitioning from state to state the important piece of information
is what happened to the ball(s) in the bottom bucket.  This can be
described by a multi-set which lists the new location(s) of the ball(s).

We can thus describe our walk by a series of multi-sets
$(T_1,T_2,\ldots,T_n)$ such that each set has $m$ elements
(when we have fewer than $m$ balls to redistribute we will indicate
no-throws by $0$).  These sets are analogous to siteswap notation for
juggling.  In particular, it can be shown that
\[
\bigcup_{i=1}^n\{i+T_i~({\rm~mod~}n)\}=\big\{\underbrace{1,\ldots,1}_{m~{\rm times}},\underbrace{2,\ldots,2}_{m~{\rm times}},\ldots,\underbrace{n,\ldots,n}_{m~{\rm times}}\big\},~~
{\rm and}~~
{1\over n}\sum_{i=1}^n\sum_{x\in T_i}x=b.
\]

In the next section we will combine the idea of this multiplex
siteswap notation with the buckets.

\section{A matrix interpretation}
One way to use the buckets to find a sequence of length $n$ that
starts in state $\alpha=\langle a_1,a_2,\ldots\rangle$ and ends
in state $\beta=\langle b_1,b_2,\ldots\rangle$ (if one exists)
is to start with the balls arranged in the buckets as dictated
by $\alpha$.
 We then modify the capacities of the buckets so that they are
 (starting at the first bucket)
\[
\underbrace{m,m,\ldots,m}_{n~{\rm times}},b_1,b_2,\ldots.
\]
Finally, take $n$ steps (such as shown in
Figure~\ref{fig:buckets}) being careful not to exceed the capacity
of any bucket and at the end, we will be forced into state
$\beta$. On the other hand, every possible way to start in
$\alpha$ and end in $\beta$ in $n$ steps can be done in this
modified buckets approach.

Finding all of the walks of length $n$ in the state diagram
between $\alpha$ and $\beta$ is thus equivalent to finding
all of the walks that can be run using this modified bucket
procedure.  This is what we will actually enumerate.  We start
with the following matrix, where $h=\max\{h(\alpha)-n,h(\beta)\}$
(this is a $0$-$1$ matrix, where any unspecified entries are $0$'s).

\bigskip

\hfil \includegraphics[scale=1]{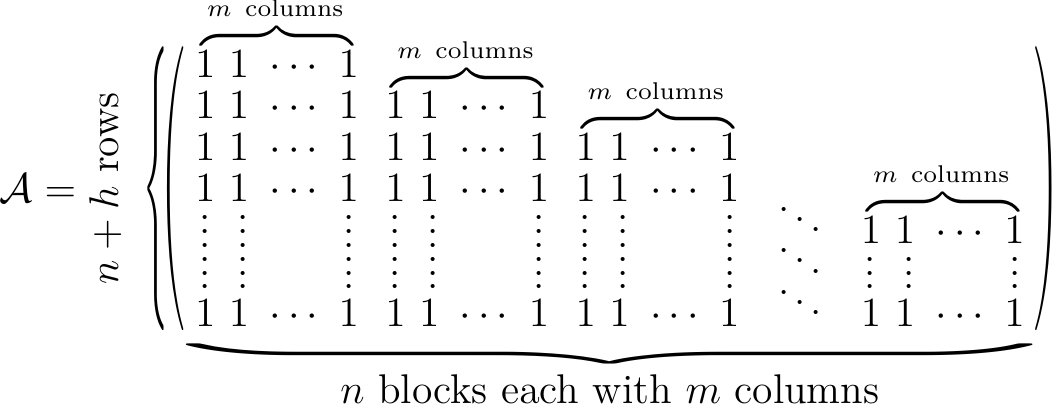} \hfil

\bigskip

Each block of $m$ columns will correspond to one transition in the
state diagram/buckets procedure.  In a block, each column corresponds
to a single element in the multi-set describing our transition.  The
first $1$ in a column corresponds to a no-throw or a throw of height
$0$, the second $1$ corresponds to a throw of height $1$, the third
$1$ to a throw of height $2$ and so on.  So we choose exactly one $1$
in each column and by reading the blocks of columns we will get the
transitions $T_i$ between states.

Each row will correspond to a bucket, and to incorporate the modified
buckets approach we specify a row sum for each row.  The row sum will
be the ``unused capacity'' of the buckets. Beginning at the first row
and going down the row sums will be
\[
m-a_1,m-a_2,\ldots,m-a_n,b_1-a_{n+1},\ldots,b_h-a_{n+h}.
\]

As an example, suppose that we wanted a walk of length $4$ that starts
and ends in the state $\langle 1,2\rangle$ in the state diagram shown
in Figure~\ref{fig:state}.  Then this corresponds to choosing one $1$
out of each column in the matrix below on the left so that the row
sums are as dictated on the side of the matrix.  \vspace{-17pt}
\[
\bordermatrix{
&\cr
\it 1&1&1\cr
\it0&1&1&1&1\cr
\it2&1&1&1&1&1&1\cr
\it2&1&1&1&1&1&1&1&1\cr
\it1&1&1&1&1&1&1&1&1\cr
\it2&1&1&1&1&1&1&1&1\cr
}
\qquad\qquad
\bordermatrix{
&\cr
\it1&\!\!\boxed{1}\!\!&1\cr
\it0&1&1&1&1\cr
\it2&1&\!\!\boxed{1}\!\!&\!\!\boxed{1}\!\!&1&1&1\cr
\it2&1&1&1&1&1&1&\!\!\boxed{1}\!\!&\!\!\boxed{1}\!\!\cr
\it1&1&1&1&\!\!\boxed{1}\!\!&1&1&1&1\cr
\it2&1&1&1&1&\!\!\boxed{1}\!\!&\!\!\boxed{1}\!\!&1&1\cr
}
\]
One possible solution is shown in the matrix on the right which corresponds to the walk
\[
\begin{array}{c@{}c@{}c@{}c@{}c@{}c@{}c@{}c@{}c}
\langle 1,2\rangle&{\to}&\langle2,1\rangle&{\to}&\langle2,0,1\rangle&{\to}&\langle0,1,2\rangle&{\to}&\langle1,2\rangle.\\
&^{\{0,2\}}&&^{\{1,3\}}&&^{\{3,3\}}&&^{\{0,0\}}&
\end{array}
\]
Here we have indicated the sets $T_i$ under each transition.

To see why this works we need to show that each multi-set given by a block of columns is a valid transition in our modified buckets procedure.  Starting with the first block, the first row  sum should be $m-\alpha_1$, this indicates that we currently have an excess capacity of $m-\alpha_1$ in the bottom bucket and so we need to make $m-\alpha_1$ no-throws, i.e., throws of height $0$, and by row sum restrictions we are forced to select exactly $m-\alpha_1$ of the $1$'s on the first row.  This accounts for $m-\alpha_1$ of the columns in the first block.  The remaining $\alpha_1$ columns must have $1$'s (i.e., balls) distributed among the rows (i.e., buckets) which still have extra capacity.  Thus the first block must give a valid transition in the procedure.  After selecting the $1$'s for the first block we then update the row sums according to our choices, remove the first block and then repeat the same process $n-1$ times.

Conversely, it is easy to check that given the transitions $T_1,T_2,\ldots,T_n$ joining $\alpha$ to $\beta$, we can find a selection of $1$'s in the matrix $\cal A$ which corresponds to this walk and satisfies the row/column sum restrictions.

Finally, since multi-sets are unordered, our choice of $1$'s in columns is unique up to permutation of the columns in a block, so we may assume that the heights of the $1$'s in a block are weakly decreasing (as shown in the example above).  We now have the following general result.

\begin{lem}\label{lem:matrix}
Suppose we are given a state diagram with capacity $m$ and states $\alpha=\langle a_1,a_2,a_3,\ldots\rangle$ and $\beta=\langle b_1,b_2,b_3,\ldots\rangle$, and $h=\max\{h(\alpha)-n,h(\beta)\}$.  Then the number of walks of length $n$ starting at $\alpha$ and ending at $\beta$ is equal to the number of ways of choosing $mn$ ones out of the matrix $\cal A$ given above such that:
\begin{itemize}
\item Each column sum is $1$.
\item The row sums are (in order from first to last)
\[
m-a_1,m-a_2,\ldots,m-a_n,b_1-a_{n+1},\ldots,b_h-a_{n+h}.
\]
\item In each block of $m$ columns the height of selected $1$'s is weakly decreasing.
\end{itemize}
\end{lem}

\begin{REM}
If for some $i$, $b_i<a_{n+i}$, then one of the row sums will be negative which is impossible and thus we have no solutions for the selection of $1$'s. At the same time it is easy to see that there can be no walks in the state diagram of length $n$ joining $\alpha$ and $\beta$ by comparing their landing schedules.
\end{REM}

\begin{REM}
When $m=1$, if we ignore rows with row sum $0$, then all the rows
and column sums will be $1$.  In this case we can count the number
of walks joining two states by calculating the permanent of a
matrix.  This is similar to the approach taken by Chung and Graham
\cite{cg1}.
\end{REM}

\section{Filling the matrix}
We now count the number of ways to fill the matrix $\cal A$ according to the restrictions in Lemma~\ref{lem:matrix}.  We will demonstrate the procedure by working through an example, namely counting the number of periodic multiplex juggling sequences of length $n$ that start and end in state $\langle 3\rangle$ when $m=3$.

The first thing to observe is that when the height of $\alpha$ is
small compared to $n$ then the row sums have the following form:
\[
{{\rm initial}\atop{\rm noise}}~~\bigg|~~{{\rm string}\atop{\rm of~}m{\rm \mbox{'s}}}~~\bigg|~~{{\rm terminal}\atop{\rm noise}}
\]
We can form a recurrence based on this where we fill in the last block and then reduce it to the case when there is one fewer $m$ in the middle.

Without loss of generality we can assume that our noise at the end is a
partition of the $b$ balls so that no part is larger than $m$, i.e., we can ignore any rows with row sum $0$ and by a simple correspondence we can assume that the row sums are weakly decreasing.

For each partition $\gamma$ of $b$ with each part at most $m$, let
$x_\gamma(k)$ be the number of ways to fill the matrix ${\cal A}$
where the row sums are given by
\[
m-a_1,\ldots,m-a_{h(\alpha)},\underbrace{m,\ldots,m}_{k~{\rm times}},\gamma.
\]
In our example, there are $3$ such partitions, $3=2+1=1+1+1$, and so we will have the
 three functions $x_{3}(k)$, $x_{2,1}(k)$ and $x_{1,1,1}(k)$.

We now form a system of linear recurrences,
$x_\gamma(k)=\sum_{\delta}a_{\gamma,\delta}x_{\delta}(k-1)$, by
examining the different ways to fill in the last block of columns.
Note that the row sums corresponding to the last block are
$m,\gamma$.  The $m$ comes from the no-throws and after we have
filled in the last block it will become incorporated into the new
terminal distribution.  Thus at each stage we will decrease
the number of middle $m$'s by $1$.

For example, if we are considering $x_{2,1}$ then we have the following $6$ different ways to fill in the last block satisfying the conditions of Lemma~\ref{lem:matrix}:
\[
{\begin{array}{c@{\,\,}|@{\,\,}c@{\,\,\,}c@{\,\,\,}c}
\it 3&\!\!\boxed{1}\!\!&\!\!\boxed{1}\!\!&\!\!\boxed{1}\!\!\\
\it 2&1&1&1\\
\it 1&1&1&1
\end{array}
\atop {\to \it 2,1}}
\quad
{\begin{array}{c@{\,\,}|@{\,\,}c@{\,\,\,}c@{\,\,\,}c}
\it 3&\!\!\boxed{1}\!\!&\!\!\boxed{1}\!\!&1\\
\it 2&1&1&\!\!\boxed{1}\!\!\\
\it 1&1&1&1
\end{array}
\atop {\to \it 1,1,1}}
\quad
{\begin{array}{c@{\,\,}|@{\,\,}c@{\,\,\,}c@{\,\,\,}c}
\it 3&\!\!\boxed{1}\!\!&\!\!\boxed{1}\!\!&1\\
\it 2&1&1&1\\
\it 1&1&1&\!\!\boxed{1}\!\!
\end{array}
\atop {\to \it 2,1}}
\quad
{\begin{array}{c@{\,\,}|@{\,\,}c@{\,\,\,}c@{\,\,\,}c}
\it 3&\!\!\boxed{1}\!\!&1&1\\
\it 2&1&\!\!\boxed{1}\!\!&\!\!\boxed{1}\!\!\\
\it 1&1&1&1
\end{array}
\atop {\to \it 2,1}}
\quad
{\begin{array}{c@{\,\,}|@{\,\,}c@{\,\,\,}c@{\,\,\,}c}
\it 3&\!\!\boxed{1}\!\!&1&1\\
\it 2&1&\!\!\boxed{1}\!\!&1\\
\it 1&1&1&\!\!\boxed{1}\!\!
\end{array}
\atop {\to \it 2,1}}
\quad
{\begin{array}{c@{\,\,}|@{\,\,}c@{\,\,\,}c@{\,\,\,}c}
\it 3&1&1&1\\
\it 2&\!\!\boxed{1}\!\!&\!\!\boxed{1}\!\!&1\\
\it 1&1&1&\!\!\boxed{1}\!\!
\end{array}
\atop {\to \it 3}}
\]
By looking at the new terminal distributions it follows that
\[
x_{2,1}(k)=x_3(k-1)+4x_{2,1}(k-1)+x_{1,1,1}(k-1).
\]
Similar analysis shows that
\begin{eqnarray*}
x_{3}(k)&=&2x_{3}(k-1)+2x_{2,1}(k-1)\qquad\mbox{and}\\
x_{1,1,1}(k)&=&x_3(k-1)+3x_{2,1}(k-1)+4x_{1,1,1}(k-1).
\end{eqnarray*}

We can rewrite this in matrix form as
\begin{equation}\label{eq1}
\left(\begin{array}{c}x_3(k)\\x_{2,1}(k)\\x_{1,1,1}(k)\\\end{array}\right)=
\left(\begin{array}{ccc}2&2&0\\1&4&1\\1&3&4\end{array}\right)
\left(\begin{array}{c}x_3(k-1)\\x_{2,1}(k-1)\\x_{1,1,1}(k-1)\\\end{array}\right).
\end{equation}
In general, we will have that
$\big(x_\gamma(k)\big)=A\big(x_\gamma(k-1)\big)$ where
$A_{\gamma,\delta}=a_{\gamma,\delta}$ as given above. Note that the
matrix $A$ will be independent of our choice of $\alpha$ and
$\beta$, and depends only on $b$ and $m$. Varying $\alpha$ will
change the initial conditions and varying $\beta$ will change which
of the $x_\gamma$ we are interested in for our recurrence.

For our problem, we are interested in $x_3$ as our terminal
state is $\langle3\rangle$ (i.e., corresponding to the partition $3$).
We want to transform our first-order system of  linear recurrences
into a single recurrence for $x_3$.
 Manipulating the recurrences in \eqref{eq1} it can be shown that
\begin{equation}\label{recurs}
x_3(k+3)=10x_3(k+2)-27x_3(k+1)+20x_3(k).
\end{equation}
This gives a recurrence for the number of periodic juggling sequence
of length $k$ for $k$ sufficiently large (i.e., we have shifted our
sequences $x_\gamma(k)$ past the initial noise; when counting
periodic juggling sequences we need to remember to account for this
shift).

\begin {REM}
The characteristic polynomial of the matrix in \eqref{eq1} is $x^3-10x^2+27x-20$,
the same coefficients as in \eqref{recurs}.  This is a consequence of the
Cayley-Hamilton Theorem.  Namely if the characteristic polynomial of the
matrix $A$ is $x^r+q_1x^{r-1}+\cdots+q_r$ then
\begin{eqnarray*}
\big(0\big)&=&\big(A^r+q_1A^{r-1}+\cdots+q_r I\big)\big(x_\gamma(k)\big)\\
&=&A^r\big(x_\gamma(k)\big)+q_1A^{r-1}\big(x_\gamma(k)\big)+\cdots+q_r\big(x_\gamma(k)\big)\\
&=&\big(x_\gamma(k+r)\big)+q_1\big(x_\gamma(k+(r-1))\big)+\cdots+q_r\big(x_\gamma(k)\big)\\
&=&\big(x_\gamma(k+r)+q_1x_\gamma(k+(r-1))+\cdots+q_rx_\gamma(k)\big).
\end{eqnarray*}
The recursion that the characteristic polynomial of $A$ gives is
universal for a fixed $b$ and $m$ in the following sense. Let $a(n)$
be the number of walks of length $n$ joining state $\alpha$ to state
$\beta$. Then for $n$ sufficiently large, $a(n)$ satisfies the
recursion given by the characteristic polynomial of $A$ (which is
independent of $\alpha$ and $\beta$).
\end{REM}

It remains to calculate enough initial terms to begin the
recursion.  Using Lemma~\ref{lem:matrix} this can be handled by brute force or some careful
case analysis.  In general, we will need to calculate the first
$h(\alpha)+r-1$ terms where $r$ is the number of partitions of $b$
with each part at most $m$.  The first $h(\alpha)-1$ terms are to
handle the initial noise caused by $\alpha$ and the next $r$ terms
are to help start the recursion.  A calculation shows that the
sequence counting the number of periodic juggling sequences that
start and end at $\langle 3 \rangle$ starts $1,4,20$. Now applying
the recursion we get the sequence
\[
1,4,20,112,660,3976,24180,147648,903140,\ldots.
\]
With the recursion and the initial values it is then a simple matter to
derive a generating function for the number of juggling sequences of
period $n$.  The generating function for this series and several others
are given in Table~1.

\begin{table}[htb]\label{table:per}
\centering
\begin{tabular}{|c|c|l|c|}\hline
{State}&$m$&{\hfil Initial~terms \hfil}&{Generating~Function}\\
\hline \hline $\langle 2 \rangle$& $2$&
$1,3,10,35,125,450,1625,5875,\ldots$&${\displaystyle x-2x^2\over
\displaystyle1-5x+5x^2}$ \\ \hline $\langle 1,1 \rangle$& $2$&
$1,3,11,40,145,525,1900,6875,\ldots$&${\displaystyle
x-2x^2+x^3\over \displaystyle1-5x+5x^2}$ \\ \hline $\langle 2,1
\rangle$& $2$& $1,4,22,124,706,4036,23110,132412,\ldots    $&$
{\displaystyle x-4x^2+3x^3\over \displaystyle1-8x+13x^2}$ \\
\hline $\langle 1,1,1 \rangle$& $2$&
$1,3,18,105,606,3483,19986,114609,\ldots    $&$ {\displaystyle
x-5x^2+7x^3\over \displaystyle1-8x+13x^2}$ \\ \hline $\langle 2,2
\rangle$& $2$& $1,3,21,162,1305,10719,88830,739179,\ldots $&$
{\displaystyle
x-11x^2+33x^3-27x^4\over \displaystyle1-14x+54x^2-57x^3}$ \\
\hline $\langle 3 \rangle$& $3$& $1,4,20,112,660,3976,24180,147648
,\ldots    $&$ {\displaystyle x-6x^2+7x^3\over \displaystyle
1-10x+27x^2-20x^3}$ \\ \hline $\langle 2,1 \rangle$& $3$&
$1,5,30,182,1110,6786,41530,254278,\ldots    $&$ {\displaystyle x
- 5x^2 + 7x^3 - 3x^4\over \displaystyle 1-10x+27x^2-20x^3}$ \\
\hline
\end{tabular}
\caption{Number of periodic sequences of length $n$ for some starting states.}
\end{table}

\begin{REM}
Instead of calculating the first $h(\alpha)+r-1$ terms we can instead calculate the first $h(\alpha)-1$ terms and then compute the $r$ values for the $x_\gamma(0)$.  In at least one case this is easier, namely when the state is $\langle b \rangle$.  Then it is easy to check that $x_\gamma(0)=1$ for all partitions $\gamma$ and quickly bootstrap our sequence.
\end{REM}

\subsection{Calculating recursion coefficients}
In this section we show how to quickly compute the recursion coefficients $a_{\gamma,\delta}$, which can then (by taking determinants) quickly give us the recursion relationship for the juggling sequences.

\begin{lem}
Let $\gamma$ and $\delta$ be partitions of $b$ with no part more than $m$, and suppose that $\gamma=(\gamma_1,\gamma_2,\ldots,\gamma_m)$ and $\delta=(\delta_1,\delta_2,\ldots,\delta_m)$ where $\gamma_i$ is the number of parts of the partition $\gamma$ of size $i$, similarly for $\delta_i$.  Then
\begin{equation}\label{coeff}
a_{\gamma,\delta}=\prod_{i=1}^m{(\gamma_i+\cdots+\gamma_m)+1-(\delta_{i+1}+\cdots+\delta_m)\choose \delta_i}.
\end{equation}
\end{lem}
\begin{proof}
To see \eqref{coeff} recall that when filling in the last set of columns of the matrix there is the partition $\gamma$ and an additional row with row sum $m$.  We then select $m$ ones such that (1) the height of the ones are weakly decreasing and (2) no row sum is violated.

This process can be made equivalent to taking the partition $\gamma$ and adding one part of size $m$ (to form a new partition $\gamma'=(\gamma_1,\ldots,\gamma_{m-1},\gamma_m+1)$) then reducing by a total of $m$ some part(s) of $\gamma'$.

The coefficient $a_{\gamma,\delta}$ is then the total number of ways that reducing some part(s) of $\gamma'$ by a total of $m$ will result in the partition $\delta$.  We can however work backwards, namely if we want a desired partition from our reduction we can ``insert'' the desired partition into $\gamma'$ (i.e., associate each part of $\delta$ with some part of $\gamma'$ which is at least as large) and then finding the difference between the insertion and $\gamma'$ (which difference will sum to $m$)  gives the reduction to use.

Thus $a_{\gamma,\delta}$ is the total number of ways that we can insert $\delta$ into $\gamma'$.  We now enumerate by inserting the partition $\delta$ backwards, namely we insert the largest parts first and then work down.  First note that here are $\delta_m$ parts of size $m$ to insert and they can be positioned into any of the $\gamma_m+1$ parts of size $m$, which can be done in
\[
{\gamma_m+1\choose\delta_m}
\]
ways.  There are then $\delta_{m-1}$ parts of size $m-1$ to insert and they can be positioned in any of the $\gamma_{m-1}+\gamma_m+1-\delta_m$ parts of size at least $m-1$ that have not yet been used, which can be done in
\[
{\gamma_{m-1}+\gamma_m+1-\delta_m\choose\delta_{m-1}}
\]
ways.  This process then continues, so that in general there will be $\delta_i$ parts of size $i$ to insert and they can be positioned in any of the $(\gamma_i+\cdots+\gamma_m)+1-(\delta_{i+1}+\cdots+\delta_m)$ parts of size at least $i$ that have not yet been used, which can be done in
\[
{(\gamma_{i}+\cdots+\gamma_m)+1-(\delta_{i+1}+\cdots\delta_m)\choose\delta_{i}}
\]
ways.  Putting it all together then gives the result.
\end{proof}

As a check for, suppose $m=2$, $\gamma=(a,b)$ and $\delta=(a-2c,b+c)$, then \eqref{coeff} gives
\[
a_{\gamma,\delta}={(a+b)+1-(b+c)\choose a-2c}{b+1\choose b+c}=
{a+1-c\choose a-2c}{b+1\choose b+c}.
\]
To get a nonzero coefficient we must have $b+1\geq b+c$ so that $c\leq 1$ and $a+1-c\geq a-2c$ so that $c\geq -1$.  Putting these in and simplifying we are left with
\[
a_{\gamma,\delta}=\left\{\begin{array}{c@{\qquad}l}\vspace{3pt}
{1\over2}a(a-1)&\mbox{if }c=1;\\\vspace{3pt}
(a+1)(b+1)&\mbox{if }c=0;\\\vspace{3pt}
{1\over2}(b+1)b&\mbox{if }c=-1;\\\vspace{3pt}
0&\mbox{otherwise.}
\end{array}\right.
\]
This can easily be verified by hand.

\section{Remarks}

\subsection{Primitive juggling sequences}
We have demonstrated a way to count the number of periodic
juggling sequences of length $n$ which start and end in a given state $\sigma$.
A related problem is counting the number of {\it primitive}\/
periodic juggling sequences.  These are special periodic juggling
sequences with the extra condition that the walk does not return
to $\sigma$ until the $n$th step.  This can be done by the
following observation of Chung and Graham \cite{cg1} (see also
\cite{gessel, concrete}): If $a(n)$ is the number of
periodic juggling sequences of length $n$ that start and end at
$\sigma$ and $F(x)=\sum_{n\geq 1} a(n)x^n$, then $b(n)$ counts the
number of primitive periodic juggling sequences of length $n$ that
start and end at $\sigma$ where
\[
\sum_{n\geq 1}b(n)x^n={F(x)\over 1+F(x)}.
\]

Applying this to the data in Table~1 we get the information about
primitive juggling sequences given in Table~2. Only one of the sequences
shown in Tables~1 or 2 was previously listed in \cite{sloane}.

\begin{table}[htb]\label{table:pri}
\centering
\begin{tabular}{|c|c|l|c|}\hline
{State}&$m$&{\hfil Initial~terms \hfil}&{Generating~Function}\\
\hline \hline $\langle 2 \rangle$& $2$&
$1,2,5,14,41,122,365,1094,\ldots$&${\displaystyle x-2x^2\over
\displaystyle1-4x+3x^2}$ \\ \hline $\langle 1,1 \rangle$& $2$&
$1,2,6,17,48,135,379,1063,\ldots$&${\displaystyle x-2x^2+x^3\over
\displaystyle 1-4x+3x^2+x^3}$ \\ \hline $\langle 2,1 \rangle$&
$2$& $1,3,15,75,381,1947,9975,\ldots    $&$ {\displaystyle
x-4x^2+3x^3\over \displaystyle1-7x+9x^2+3x^3}$ \\ \hline $\langle
1,1,1 \rangle$& $2$& $1,2,13,68,358,1871,9757,\ldots    $&$
{\displaystyle x-5x^2+7x^3\over \displaystyle1-7x+8x^2+7x^3}$
\\ \hline $\langle 2,2 \rangle$& $2$&
$1,2,16,119,934,7463,60145,\ldots    $&$ {\displaystyle
x-11x^2+33x^3-27x^4\over \displaystyle1-13x+43x^2-24x^3-27x^4}$ \\
\hline $\langle 3 \rangle$& $3$& $1,3,13,67,369,2083,11869,\ldots
$&$ {\displaystyle x-6x^2+7x^3\over \displaystyle
1-9x+21x^2-13x^3}$
\\ \hline $\langle 2,1 \rangle$& $3$&
$1,4,21,111,592,3171,17021,\ldots    $&$ {\displaystyle x - 5x^2 +
7x^3 - 3x^4\over \displaystyle 1 - 9x + 22x^2 - 13x^3 - 3x^4}$ \\
\hline
\end{tabular}
\caption{Number of primitive sequences of length $n$ for some starting states.}
\end{table}

A related open problem is to find the number of prime juggling
sequences which start and end in a given state $\sigma$.  A prime juggling
sequence corresponds to a simple cycle in the state diagram, i.e., it never visits any vertex more
than once. Note that for a primitive juggling sequence we are
allowed to visit vertices other than $\sigma$ as often as we want.

\subsection{Further directions}
One implicit assumption that we have made is that the height the
balls can be thrown to in the juggling sequence is essentially
limited only by the period. This might be unrealistic when trying to
implement the procedure for an actual juggler.  In this case we
would like to add an additional parameter which is the maximum
height a ball can be thrown.  While it is not difficult to adopt the
matrix $\cal A$ to handle this additional constraint, our recursion
method will no longer work.  For this setting, the simplest method
might be to find the adjacency matrix of the (now finite) state
diagram and take powers to calculate the number of walks.

Another implicit assumption that we have made is that the balls are identical,
it is easy to imagine that the balls are distinct and then we can ask given
an initial placement of balls and a final placement of balls how many walks
in the state diagram are there.  This problem is beyond the scope of the methods given here.  However, Stadler \cite{stadlet} has had some success in this direction (using different methods than the ones presented here), he was able to derive an expression involving Kostka numbers enumerating the number of such sequences, as well as several other related sequences.

It would be interesting to know for each $m$, $n$ and $b$, which
states $\sigma$ have the largest number of (primitive)
$\sigma$-juggling sequences of length $n$. When $m = 1$, then it
would seem that the so-called ground state $\langle 1,1,\ldots,1 \rangle$ does.
However, for larger values of $m$, it is not so clear what to guess.


As can be seen there are still many interesting open problems
concerning the enumeration of multiplex juggling sequences.

\end{document}